\documentclass[12pt]{amsart}
\usepackage{amsmath,amsfonts,amssymb,amscd,amsthm,calc,enumerate, graphicx}

\newtheorem{theorem}{Theorem}[section]

\newtheorem{lemma}[theorem]{Lemma}
\newtheorem{corollary}[theorem]{Corollary}

\theoremstyle{definition}
\newtheorem{definition}[theorem]{\bf Definition}

\newcommand{\restr}{\mbox{\raisebox{.5mm}{$\upharpoonright$}}}

\newcommand{\la}{\langle}
\newcommand{\ra}{\rangle}

\newcommand{\bigset}[1]{\big\{ #1 \big\}}

\renewcommand{\leq}{\leqslant}
\renewcommand{\geq}{\geqslant}
\renewcommand{\le}{\leqslant}
\renewcommand{\ge}{\geqslant}

\newcommand{\vph}{\varphi}

\newcommand{\M}{{\mathfrak M}}

\newcommand{\B}{\mathcal{B}}

\newcommand{\K}{\mathcal{K}}

\renewcommand{\P}{\mathcal{P}}

\newcommand{\IPC}{{\sf IPC}}
\newcommand{\Log}{{\sf L}}
\newcommand{\KC}{{\sf KC}}

\newcommand{\wlem}{{w.l.e.m.}}

\newcommand{\Op}{{\sf Op}}
\newcommand{\op}{{\rm op}}
\newcommand{\Alg}{{\sf Alg}}
\newcommand{\Kr}{{\sf Kr}}
\newcommand{\Th}{{\rm Th}}
\newcommand{\join}{+}

\renewcommand{\Join}{{\textstyle \sum}}

\renewcommand{\phi}{\varphi}

\begin{document}

\title[Generalizations of the Weak Law of the Excluded Middle]
{Generalizations of the Weak Law of the Excluded Middle}

\author[A. Sorbi]{Andrea Sorbi}
\address[Andrea Sorbi]{University of Siena \\
Dipartimento di Scienze Matematiche ed Informatiche ``Roberto Magari''\\
Pian dei Mantellini 44, 53100 Siena, Italy. } \email{sorbi@unisi.it}
\author[S. A. Terwijn]{Sebastiaan A. Terwijn}
\address[Sebastiaan A. Terwijn]{Radboud University Nijmegen\\
Department of Mathematics\\
P.O. Box 9010, 6500 GL Nijmegen, the Netherlands.
} \email{terwijn@math.ru.nl}

\begin{abstract}
We study a class of formulas generalizing the
weak law of the excluded middle, and
provide a characterization of these formulas
in terms of Kripke frames and Brouwer algebras.
We use these formulas to separate logics corresponding to
factors of the Medvedev lattice.
\end{abstract}


\date{\today}

\maketitle

\section{The weak law of the excluded middle}

Let $\IPC$ denote the intuitionistic propositional calculus.
The weak law of the excluded middle (\wlem\ for short) is the principle
\begin{equation} \label{wlem}
\neg p \vee \neg\neg p.
\end{equation}
We view this as an axiom {\em schema\/}, in which we can substitute any
formula for the variable $p$.
Consider the logic $\IPC + \neg p \vee \neg\neg p$, that is, the
closure under deductions and substitutions of $\IPC$ and the \wlem\
The logic $\IPC + \neg p \vee \neg\neg p$ has been studied
extensively, and is known in the literature under various names.
It has been called
\begin{itemize}
\item {\em the logic of the weak law of the excluded middle\/} by Jankov,
\item {\em Jankov logic\/} by various Russian authors,
\item {\em De Morgan logic\/} by various American authors,
\item {\em testability logic\/} by some others, and
\item $\KC$ by still many others.
\end{itemize}
The term {\em principle of testability\/} for $\neg p \vee \neg\neg p$ goes
back to Brouwer himself.
In \cite[p80]{Brouwer1964} he writes (our comment in brackets):
\begin{quote}
``Another corollary of the simple principle of the
excluded third [i.e.\ $\tau\vee\neg \tau$] is the
{\em simple principle of testability},
saying that every assignment $\tau$
of a property to a mathematical entity can be {\em tested\/},
i.e.\ proved to be either non-contradictory [$\neg\neg\tau$]
or absurd [$\neg\tau$].''
\end{quote}
Apparently the name $\KC$ derives from
Dummett and Lemmon \cite{DummettLemmon}, who used $\sf LC$ to denote the
``linear calculus'', and K alphabetically follows~L,
hence~$\KC$.

In this paper we will study the following sequence $\{\vph_{k}\}_{k \ge 1}$
of formulas generalizing the \wlem:
\begin{definition}
Let $\vph_1 = \neg p \vee \neg\neg p$, and
for every $k>1$ define
\begin{equation} \label{eq:phik}
\vph_k =
\bigvee_{i\neq j} \big(\neg p_i \rightarrow \neg p_j\big) \vee
\neg(\neg p_1\wedge\ldots \wedge \neg p_k)
\end{equation}
(where $1 \le i,j \le k$).
\end{definition}
Notice that the formula $\vph_1$ can be seen as a special case of $\vph_k$:
indeed, $\vph_k$ is equivalent over $\IPC$ to
\begin{equation}
\neg p_1 \vee \ldots \vee \neg p_k \vee
\bigvee_{i\neq j} \big(\neg p_i \rightarrow \neg p_j\big) \vee
\neg(\neg p_1\wedge\ldots \wedge \neg p_k)
\end{equation}
because $\neg p_i$ implies $\neg p_j \rightarrow \neg p_i$ in $\IPC$.
Then $\vph_1$ is the special case $k=1$.

Also note that $\IPC$ proves $\vph_k \rightarrow \vph_{k+1}$ for every $k\geq 1$.
This follows for example from Theorem~\ref{thm:char}, or from
Theorem~\ref{thm:algebraic} below.

Below, we will study the logics $\IPC + \vph_k$, which again is the
deductive closure of $\IPC$ and the axiom schema $\vph_k$.
In particular $\IPC + \vph_k$ proves any substitution instance of~$\vph_k$.

\section{Kripke semantics}

In this section we characterize the formulas $\vph_k$
in \eqref{eq:phik} in terms of Kripke frames,
and relate them to a class of formulas introduced by
Smorynski~\cite{Smorynski}.

We briefly recall some elementary notions about Kripke semantics.
For unexplained terminology about Kripke frames and models we refer
the reader to~\cite{ChagrovZakharyaschev} or~\cite[p67]{Gabbay}.

A \emph{Kripke frame\/} $\langle K, R \rangle$ is a nonempty set $K$,
partially ordered by an \emph{accessibility relation\/} $R$.
Throughout this paper, we will work with Kripke frames
that have a {\em root\/}, that is,
a least element with respect to~$R$,
though this is not standardly part of the definition.
As usual, we distinguish between models and frames:
A \emph{Kripke model} $\la K,R,V\ra$ is a Kripke frame
together with a {\em valuation\/} $V$, that associates with
every variable $p$ a set $V(p)\subseteq K$, such that if
$x\in V(p)$ and $xRy$ then $y\in V(p)$ for every $x$ and~$y$.
Now the forcing relation $x \Vdash\vph$, with $x\in K$ and $\vph$
a formula, is defined by
\begin{itemize}
\item $x \Vdash p$ if $x\in V(p)$;
\item $x \Vdash \phi \wedge \psi$ if and only if $x \Vdash \phi $
and $x \Vdash  \psi$;
\item $x \Vdash \phi \vee \psi$ if and only if $x \Vdash \phi $
or $x \Vdash  \psi$;
\item
$x \Vdash \phi \rightarrow \psi$ if and only if for every $y$ with $x\, R y$, if
$y \Vdash \phi $ then $y \Vdash  \psi$;
\item
$x \Vdash \neg \phi$ if and only if there is no $y$ with $x\, R y$ and $y \Vdash \phi $.
\end{itemize}
A formula $\phi$ \emph{holds\/} in a frame $K$, denoted by $K \models \phi$,
if $K \Vdash \phi$ (meaning that $x \Vdash \phi$ for every $x \in K$), for
every valuation $V$ on the frame. A logic $\Log$ is {\em complete with
respect to\/}, or {\em characterizes\/}, a class of frames $\K$ if a formula
is derivable in $\Log$ if and only if it holds on every frame in $\K$.

\begin{definition}
A Kripke frame with accessibility relation $R$ has {\em topwidth\/} $k$ if it
has $k$ maximal nodes $x_1,\ldots,x_k$ such that for every $y\in K$ there is
an $i$ with $y R x_i$
\end{definition}

Following Jankov~\cite{Jankov}, Gabbay~\cite[p67]{Gabbay} showed that the logic
$\IPC+ \neg p \vee \neg\neg p$ is complete with respect
to the class of Kripke frames of topwidth~1.
Smorynski~\cite{Smorynski} introduced, for every $k\ge 1$, the formula
\begin{equation}\label{Smor}
\sigma_{k}=
\bigwedge_{0\leq i < j \leq k} \neg\big(\neg p_i \wedge \neg p_j\big)
\rightarrow
\bigvee_{0 \le i \le k} \Big(\neg p_i \rightarrow \bigvee_{j\neq i} \neg p_j\Big)
\end{equation}
and showed that the logic $\IPC+\sigma_{k}$ characterizes
the class of Kripke frames of topwidth at most $k$ (henceforth
we refer to this result as Smorynski's Completeness Theorem).
In particular,
$\IPC$ proves that  $\sigma_k \rightarrow \sigma_{k+1}$
and $\IPC+\sigma_{1}$ coincides with the logic of the \wlem.
Note that $\vph_k$ has $k$ variables and $\sigma_k$ has $k+1$.
The relation between these formulas is sorted out below.

We now turn to a characterization of the formulas $\vph_k$ in~(\ref{eq:phik})
in terms of Kripke frames. We start with some preliminaries about canonical
models. For more on canonical models we refer to~\cite{ChagrovZakharyaschev}.
The \emph{canonical model\/} $K$ of a logic $\Log$ containing $\IPC$ consists
of tableaux, that is, pairs $(\Gamma, \Delta)$ of sets of formulas,
satisfying the following properties:\footnote{%
Gabbay~\cite{Gabbay} uses saturated sets of formulas to define the canonical
model, which is similar but different.}
\begin{enumerate}[(i)]
\item \label{Kpr1} $(\Gamma, \Delta)$ is {\em consistent\/} with $\Log$,
    meaning that for no $\phi_1, \ldots \vph_n\in \Delta$, $\Gamma$
    proves $\vph_1\vee\ldots\vee \vph_n$ over $\Log$,
\item \label{Kpr2}
$(\Gamma, \Delta)$ is {\em maximal\/} in the sense that
$\Gamma\cup\Delta$ is the set of all formulas.
\end{enumerate}
The accessibility relation $R$ in the canonical model is
defined by
$$
(\Gamma,\Delta)\, R \,(\Gamma',\Delta') \Longleftrightarrow
\Gamma\subseteq \Gamma' \Longleftrightarrow \Delta \supseteq \Delta'.
$$
This defines the canonical frame, and to make it into
a model it is defined that every atomic formula in $\Gamma$
is forced in the node $(\Gamma,\Delta)$.
It is a basic property of $K$ that for every node
$(\Gamma, \Delta)$ and every formula $\vph$,
$$
(\Gamma, \Delta) \Vdash \vph \Longleftrightarrow \vph\in\Gamma.
$$
Note that it follows from properties (\ref{Kpr1}) and (\ref{Kpr2})
that $\Gamma$ is closed under $\Log$-provability.

\begin{lemma} \label{lem:compdir1}
Suppose $K$ is a Kripke frame of topwidth $n+1$ in which
$\vph_k$ does not hold.
Then $\binom{n}{\lfloor n/2\rfloor} \geq k$.
\end{lemma}
\begin{proof}
Under the assumptions,
we prove that the power set $\P(\{1,\ldots,n\})$ has an antichain of
size~$k$. The lemma then follows from Sperner's Theorem,
(\cite{Sperner1928}; cf.\ also \cite{Proofs}) stating
that $\binom{n}{\lfloor
n/2\rfloor}$ is the greatest number $k$ for which there is an antichain
of $k$ pairwise incomparable subsets of $\{1,\ldots, n\}$.

Since there is a model on the frame $K$ that falsifies $\vph_k$, there must
be a maximal node in which $\neg p_1\wedge\ldots\wedge \neg p_k$ holds. This
leaves $n$ nodes to falsify all implications $\neg p_i \rightarrow \neg p_j$
with $i\neq j$. Label these nodes by $1,\ldots,n$. Let $S_i \subseteq
\{1,\ldots,n\}$ be the set of nodes where $p_i$ holds, with $i = 1,\ldots,
k$. Then the sets $S_i$ form an antichain since for every pair $i\neq j$
there is a node that falsifies $\neg p_i \rightarrow \neg p_j$, hence in
which $p_i$ and $\neg p_j$ hold.
\end{proof}

\begin{lemma} \label{lem:compdir2}
Suppose $(\Gamma_1,\Delta_1),\ldots,(\Gamma_n,\Delta_n)$ are distinct
maximal nodes in the canonical model of $\Log$. Then for every $S\subseteq
\{1,\ldots,n\}$ there is a formula $A$ such that $A \in \Gamma_j$ if and only
if $j \in S$.
\end{lemma}
\begin{proof}
By maximality, the $\Gamma_i$ are pairwise $\subseteq$-incomparable, hence
for every $i \ne j$, there is a formula $A_{i,j}\in \Gamma_i - \Gamma_j$.
Hence, taking, $A_i=\bigwedge_{j \ne i} A_{i,j}$,
for every $i$,  it is easy to see that $(\Gamma_i, \Delta_i)\Vdash A_i
\rightarrow \neg A_j$
for every $i\neq j$. Now let $A = \bigvee_{j \in S} A_j$.
\end{proof}

\begin{theorem} \label{thm:char}
$\IPC + \vph_k$ is complete with respect to
the class of Kripke frames of topwidth at most~$n$,
where $n$ is minimal such that
$$
\binom{n}{\lfloor n/2\rfloor} \geq k.
$$
\end{theorem}
\begin{proof}
For the right-to-left implication,
suppose $K$ is a frame of topwidth $m+1\le n$
in which $\vph_k$ does not hold. Then by Lemma~\ref{lem:compdir1},
$\binom{m}{\lfloor m/2\rfloor} \geq k$,
hence by minimality of $n$ we have $m\geq n$, a contradiction.
Hence any frame of topwidth $l\leq n$ satisfies $\vph_k$.

For the converse direction, we have to show that if $\vph$ is a formula that
$\IPC+\vph_k$ does not prove, then there is a Kripke frame of topwidth at
most~$n$,
where $n$ and $k$ are related as in the statement of the theorem, in which
$\vph$ does not hold, i.e.\ there is a model on this frame on which $\vph$
does not hold. We show that a part of the canonical model of $\IPC+\vph_k$
has this property.

Now if $\vph$ is not provable in $\IPC+\vph_k$, then
its negation is consistent, hence $\neg \vph$ is forced
at some node $t=(\Gamma, \Delta)$ of the canonical model,
and $\vph$ does not hold in $t$.
Let $K^t$ denote the part of $K$ that is $R$-reachable from $t$.
We prove that $K^t$ has the required property.

First we note that every node in $K$ is below an $R$-maximal
one: every path in $K$ has an upper bound (by taking unions on the
first coordinate and intersections on the second),
hence an application of Zorn's lemma gives a maximal element
above any node in $K$.

We now show that $K^t$ has at most $n$ $R$-maximal nodes.
Suppose for a contradiction that there exist at least $n+1$
distinct maximal nodes
\[
(\Gamma_1,\Delta_1),\ldots,(\Gamma_{n+1},\Delta_{n+1}).
\]
Since $\binom{n}{\lfloor n/2\rfloor} \geq k$ there is an
antichain $S_1,\ldots, S_k$ in $\P(\{1,\ldots,n\})$ of size~$k$.
For every $S_i$, with the help of Lemma~\ref{lem:compdir2}
choose a formula $A_i$ such that
\begin{equation} \label{Ai}
A_i \in \Gamma_j \Longleftrightarrow j\in S_i
\end{equation}
and such that $A_i \notin \Gamma_{n+1}$.
Note that by maximality it follows from \eqref{Ai} that
$$
\neg A_i \in \Gamma_j
\Longleftrightarrow A_i \notin \Gamma_j
\Longleftrightarrow j\notin S_i.
$$
But now we can prove that $\vph_k$ is not forced in $t$:
First $t \not \Vdash \neg(\neg A_1\wedge\ldots\wedge \neg A_k)$
because
$(\Gamma_{n+1},\Delta_{n+1}) \Vdash \neg A_1\wedge\ldots\wedge\neg A_k$
by choice of $A_i$.
Also $t \not \Vdash \neg A_i \rightarrow \neg A_{i'}$ for
every $i\neq i'$ with $i,i'\leq k$. Namely,
the elements $S_i$ and $S_{i'}$ of the antichain are incomparable,
hence $j\in S_{i'} - S_i$ for some $j\in\{1,\ldots,n\}$.
Thus, by definition of $A_i$, we have
$A_{i'}\in \Gamma_j$ and $\neg A_i \in \Gamma_j$, and hence
$(\Gamma_j,\Delta_j)\Vdash \neg A_i \wedge A_{i'}$.
So we see that $t$ does not force the formula
$\vph_k(A_1,\ldots,A_k)$ obtained from $\vph_k$ by
substituting $A_i$ for every variable $p_i$.
But then it follows that $t\not\Vdash \vph_k$, for if
$t\Vdash \vph_k$ then $t$ would also force $\vph_k(A_1,\ldots,A_k)$
because we work over the logic $\IPC+\vph_k$,
which by definition proves every substitution instance of~$\vph_k$.
\end{proof}

A logic $\Log$ is called {\em canonical\/} if every formula of $\Log$ holds in the
canonical frame of $\Log$. Note that the proof of Theorem~\ref{thm:char} shows
that the logics of $\vph_k$ are canonical in this sense.

Following \cite[p69]{Gabbay}, a condition $F$  on a partially ordered set
$\langle K,R,0 \rangle$ with least element $0$, is \emph{absolute} if it can
be formulated in higher order language (with symbols for $R, 0, =$), and for
every $\langle K,R,0\rangle$ satisfying $F$, there exists a finite $K_0
\subseteq K$ such that for every $K'$, with $K_0 \subseteq K' \subseteq K$,
we have that also $\langle K', R\restr K', 0\rangle$ satisfies $F$. It is
known, see e.g.~Gabbay~\cite[p69]{Gabbay}, that if $\Log$ is an intermediate
logic which characterizes a class of Kripke frames, consisting of exactly the
frames satisfying an absolute condition $F$, then $\Log$ also characterizes
the class of finite Kripke frames satisfying~$F$. An intermediate logic
$\Log$ is said to have the \emph{finite model property}, if for every $\phi$
with $\phi\notin \Log$, there exists a finite Kripke model which does not
satisfy $\phi$. By a classical theorem of Harrop (\cite{Harrop1958}; see also
\cite[p.~266]{Gabbay}), if an intermediate logic $\Log$ has the finite model
property and is finitely axiomatizable, then $\Log$ is decidable. Therefore
we have:

\begin{theorem}\label{thm:fmp}
Each $\IPC + \phi_k$ is complete with respect to the class of
{\em finite\/} Kripke frames with topwidth at most $n$, where $n$
is least such that $\binom{n}{\lfloor n/2\rfloor} \geq k$.
Moreover,  $\IPC + \phi_k$ is decidable.
\end{theorem}
\begin{proof}
The claim follows by the above quoted remark and the fact the
condition of being a Kripke frame with topwidth at most $n$, and $n$ least
such that  $\binom{n}{\lfloor n/2\rfloor} \geq k$, is absolute.
\end{proof}

Finally, we have the following additional characterization of
$\IPC+\vph_{k}$:

\begin{corollary}\label{cor:characterization}
$\IPC+\phi_{k}=\IPC+\sigma_{n}$,
for all $n$ and $k$ such that $n$ is minimal with
$\binom{n}{\lfloor n/2\rfloor} \geq k$.
\end{corollary}
\begin{proof}
This follows from Theorem~\ref{thm:char} and Smorynski's Completeness Theorem.
\end{proof}

Notice that the sequence of logics $\IPC+\phi_{k}$ is decreasing, but not
strictly decreasing, with respect to inclusion. Namely, if $k_{1} < k_{2}$
and $n$ is the least such that $\binom{n}{\lfloor n/2\rfloor} \geq k_{1}$,
but $n$ is also the least such that $\binom{n}{\lfloor n/2\rfloor} \geq
k_{2}$, then
\[
\IPC+\vph_{k_{1}}=\IPC+\vph_{k_{2}}=\IPC+\sigma_{n}.
\]

\section{Algebraic semantics} \label{sec:Brouwer}

A \emph{Brouwer algebra} is an algebra $\langle L, +, \times, \rightarrow,
\neg, 0,1\rangle$ where $\langle L, +, \times, 0,1\rangle$ is a bounded
distributive lattice (with $+$ and $\times$ denoting the operations of $\sup$
and $\inf$, respectively)
and $\rightarrow$ is a binary operation satisfying
\begin{equation}\label{eqn:arrow}
b \le a + c \Leftrightarrow a\rightarrow b \le c,
\end{equation}
or, equivalently,
\[
a\rightarrow b= \text{least }\{c: b \le a +c\},
\]
and $\neg$ is the unary operation, given by $\neg a= a \rightarrow 1$. A
Brouwer algebra $L$ \emph{satisfies} a propositional formula $\sigma$
(denoted by $L \models \sigma$) if whatever substitution of elements of $L$
in place of the propositional variables of $\sigma$ (interpreting the
connectives $\lor$, $\wedge$, $\rightarrow$, $\neg$ with the operations
$\times$, $+$, $\rightarrow$, $\neg$, respectively) yields the element~$0$.
(Note that this definition of truth is dual to that in a Heyting algebra; see
also the remarks on Heyting algebras below.) Let
$$
\Th(L)=\{\sigma: L\models \sigma\}.
$$
It is well known that $\IPC \subseteq \Th(L)$, for every Brouwer algebra $L$.
An intermediate logic $\Log$ is \emph{complete with respect to} a class
of Brouwer algebras, if for every formula $\sigma$, $\Log$ derives
$\sigma$ if and only if every algebra in the class satisfies~$\sigma$.

Recall that in a distributive lattice $L$, we have that an element $a \in L$
is join-irreducible if and only if $a \le x + y$ implies $a \le x$ or $a \le
y$, for every $x,y \in L$. Thus if $L$ is a Brouwer algebra, $b \in L$ with
$b=\sum X$, where $X$ consists of join-irreducible elements, then for every
$a \in L$,
\begin{equation}\label{arrow}
a \rightarrow b= \sum \{x \in X: x \not \le a\}:
\end{equation}
This follows from the fact that $b \le a + y$, where $y= \sum \{x \in X: x
\not \le a\}$, and  by join-irreducibility of each element of $X$, we have
that $x \le c$ for every $c$ such that $b \le a + c$ and every $x \in X$ such
that $x \nleq a$. Thus $y$ is the least such that $b \le a + y$.
Finally, if $X$ is an antichain of join-irreducible elements
in a distributive lattice, and $I, J \subseteq X$ are finite sets, then
\begin{equation}\label{incomp}
\sum I \le \sum J \Leftrightarrow I \subseteq J.
\end{equation}

Recall the following well-known construction (see \cite{Fitting}) which
associates with every Kripke frame a Brouwer algebra, whose
identities coincide with the formulas that hold in the frame. Let $K$ be a
given Kripke frame, with accessibility relation $R$:
a subset $A \subseteq K$ is \emph{open}, if for every
$x,y\in K$ we have that $x \in A$ and $x R y$ then $y \in A$.
Let $\Op(A)$ be the collection of open subsets of $K$.

\begin{lemma}[\cite{Fitting}]\label{lem:Fitting}
The distributive lattice $\Alg(K)=\langle \Op(K), +, \times, \rightarrow 0,
1\rangle$ is a Brouwer algebra, where $A+B=A\cap B$, $A \times B=A \cup B$,
$A\rightarrow B=\{x\in K: (\forall y \in K)
[x R y \wedge y \in A \Rightarrow y \in B]\}$,
$0=K$, and $1=\emptyset$. Moreover
$$
\{\phi: K \models \phi\}=\{\phi: \Alg(K)\models \phi\}.
$$
\end{lemma}

\begin{proof}
See \cite{Fitting}. In fact, the theorem in \cite{Fitting} is formulated in
terms of Heyting algebras. Recall that $L$ is a Heyting algebra if the dual
$L^{\op}$ is a Brouwer algebra. If $L$ is a Heyting algebra, we write $L
\models^H \sigma$, if $L^{\op} \models \sigma$. In \cite{Fitting} it is shown
that the collection of open sets together with the operations $+=\cup$,
$\times=\cap$, $0=\emptyset$, $1=K$, and
$$
A\rightarrow B=\bigset{x\in K: (\forall y \in K)
[x R y \wedge y \in A \Rightarrow y \in B]},
$$
is a Heyting algebra which satisfies the same formulas as $K$. To prove our
result, given a frame $K$, apply Fitting's theorem to get a Heyting algebra,
and then take its dual: the claim then follows from the obvious fact that the
formulas satisfied (under $\models$) by a Brouwer algebra are the same as the
ones satisfied (under $\models^H$) by its dual Heyting algebra.
\end{proof}

Conversely, given a Brouwer algebra $L$ with meet-irreducible $0$,
let $I(L)$ be the collection of prime ideals of $L$,
which becomes a Kripke frame $\Kr(L)=\langle I(L), \subseteq \rangle$.
(Note that $\Kr(L)$ satisfies our assumption that all Kripke frames
have a root, since $0\in L$ is meet-irreducible, so that $\{0\}$ is
a prime ideal.)

\begin{lemma}\cite{Ono-Kripke}\label{lem:Ono}
For every Brouwer algebra $L$, we have
$$
\{\phi: L \models \phi\}\subseteq \{\phi: \Kr(L) \models \phi\}.
$$
Moreover, equality holds if $L$ is finite.
\end{lemma}

\begin{proof}
See \cite{Ono-Kripke}. Again, a few words may be spent on the proof, since
\cite{Ono-Kripke} uses Heyting algebras instead of Brouwer algebras. So,
suppose we are given a Brouwer algebra $L$, take its dual $L^{\op}$, which is
a Heyting algebra, and then use \cite{Ono-Kripke} to conclude that $\langle
F(L^{\op}), \subseteq\rangle$ (where $F(L^{\op})$ is the collection of prime
filters of $L^{\op}$) is a Kripke frame $K$ that satisfies $\{\phi: L^{\op}
\models^H \phi\}\subseteq \{\phi: K \models \phi\}$, with equality if
$L^{\op}$ is finite. The claim then follows from the fact that $\{\phi: L^{\op}
\models^H \phi\}=\{\phi: L \models \phi\}$, and $F(L^{\op})$ is order
isomorphic to $I(L)$ under $\subseteq$, as easily follows from recalling that
in a distributive lattice $L$, for every $X\subseteq L$, $X$ is a prime
filter if and only if $L-X$ is a prime ideal.
\end{proof}

Theorem~\ref{thm:fmp} has the following algebraic counterpart:

\begin{theorem} \label{thm:algebraic}
$\IPC + \vph_k$ is complete with respect to the class of all finite Brouwer
algebras $L$ with meet-irreducible $0$ and at most $n$ coatoms, where $n$ is
minimal such that $ \binom{n}{\lfloor n/2\rfloor} \geq k$.
\end{theorem}

\begin{proof}
The proof follows from Theorem~\ref{thm:char},
Lemma~\ref{lem:Fitting}, Lemma~\ref{lem:Ono},
together with the following observations:

\begin{enumerate}
\item If $K$ has topwidth $n$, then $\Alg(K)$ has $n$ coatoms: indeed,
    for every maximal element $x$ in the frame, the singleton $\{x\}$ is
    open, and this is clearly a coatom in $\Alg(K)$; moreover the coatoms
    in $\Alg(K)$ are all of this form.

\item If a finite Brouwer algebra $L$  has $n$ coatoms, then $\Kr(L)$ is
    of topwidth $n$:
    indeed, in a finite Brouwer algebra $L$, the ideals
    generated by the coatoms are prime and contain all other prime
    ideals, generated by meet-irreducible elements. In other words the
    coatoms correspond exactly to the maximal elements in $\Kr(L)$.
\end{enumerate}
Finally,  notice that, for every Kripke frame $K$,
$\Alg(K)$ has meet-irreducible~$0$, since the Kripke frames
in this paper always have a least element.
\end{proof}

For finite Brouwer algebras, we may also describe the completeness property
in terms of join-irreducible elements joining to the greatest element $1$.
\begin{definition}
For every $n$, let $\mathfrak{B}_n$ denote the class of Brouwer algebras in
which the top element is the join of some antichain of $n$ join-irreducible
elements.
\end{definition}
\noindent
Notice that in any distributive lattice, if $\sum X=\sum Y$, where $X, Y$
are finite antichains of join-irreducible elements, then it follows from
(\ref{incomp}) that $X=Y$. Thus, in a finite distributive lattice $L$, or
more generally in a distributive lattice $L$ having the finite descending
chain condition (see e.g.~\cite[Theorem~III.2.2]{Balbes-Dwinger:Book}) each
element is the join of a unique antichain of join-irreducibles,
and thus $L$ belongs to $\mathfrak{B}_n$, for a unique $n$.

\begin{lemma}\label{lem:coatoms-joinirr}
If $L$ is a finite Brouwer algebra, then $L$ has exactly $n$
coatoms if and only if $L \in \mathfrak{B}_n$.
\end{lemma}

\begin{proof}
Suppose that $L \in \mathfrak{B}_n$ is finite, and let $b_1, \ldots, b_n$ be
the antichain of $n$ join-irreducible elements such that $1=\sum_{i=1}^n
b_i$. For every $i$, let $\hat{b_i}=\sum_{j\ne i}b_j$. We claim that each
$\hat{b_i}$ is a coatom. Indeed $\hat{b_i}<1$, as $b_i \nleq \hat{b_i}$;
moreover, assume that $\hat{b_i} \le b$, and let $b=\sum X$ where $X$ is an
antichain of join-irreducible elements. (Here we use that $L$ is finite.) By
join irreducibility, we have
\[
\{b_j: j \ne i\} \subseteq X \subseteq \{b_j: 1 \le j \le n\}
\]
thus either $\hat{b_i}=b$ or $b=1$. It follows that $L$ has at least $n$
coatoms. On the other hand, suppose that $L$ has also a coatom
$a\notin \{\hat{b_i}: 1\le i \le n\}$. Then for every~$i$,
$\hat{b_i}+ a=1$, thus $b_i \le \hat{b_i}+ a$,
hence by join irreducibility, $b_i \le a$.
This implies that $\sum_i b_i \le a$, hence $a=1$, a contradiction.

Conversely, suppose that $L$ is a finite Brouwer algebra that has $n$
coatoms. Since $L$ is finite, there exists $m$ such that
$L\in\mathfrak{B}_m$. On the other hand, the above argument shows that $m=n$,
so that $L \in \mathfrak{B}_n$.
\end{proof}

Let $\mathfrak{B}_n^\bot$ be the subclass of $\mathfrak{B}_n$, consisting of
the algebras with meet-irreducible $0$. It follows:

\begin{corollary}\label{cor:fin-char-II}
$\IPC + \vph_k$ is complete with respect to the class of finite
Brouwer algebras $\mathfrak{B}_n^\bot$, where $n$ is minimal such
that $ \binom{n}{\lfloor n/2\rfloor} \geq k$.
\end{corollary}

\begin{proof}
Immediate from Theorem~\ref{thm:algebraic}, and Lemma~\ref{lem:coatoms-joinirr}.
\end{proof}

Finally, we prove Theorem~\ref{thm:validity} below, which holds also of
Brouwer algebras that are not necessarily finite. We need a preliminary
lemma, which illustrates the range of $\neg$ in a Brouwer algebra from
$\mathfrak{B}_n$.

\begin{lemma} \label{lemman}
Let $L\in \mathfrak{B}_n$, and let $b_1, \ldots, b_n$ be an antichain of
join-ir\-red\-uc\-ible elements such that $1=b_1+ \cdots +b_n$. Then every
negation $\neg a$ in $L$ is of the form $\neg a = \Join_{i\in I} b_i$ for
some subset $I\subseteq\{1,\ldots, n\}$ (where, of course, $\neg a=0$ if
$I=\emptyset$).
In particular, $\neg b_{i}=  \Join_{j \ne i}b_{j}$.
\end{lemma}
\begin{proof}
By (\ref{arrow}) we have $\neg a= \sum_{i \in I} b_{i}$,
where $I=\{i: b_i \not\leq a\}$.
\end{proof}

\begin{theorem}\label{thm:validity}
Let $\binom{n}{\lfloor n/2\rfloor} = k$. Then the following hold:
\begin{enumerate}[\rm (i)]
\item If  $L \in  \mathfrak{B}_{m}$ and $m \le n$, then
      $L\models \vph_{k}$;
\item if  $L \in  \mathfrak{B}_{m}^{\bot}$ and $m >n$ then $L\not\models
      \vph_{k}$.
\end{enumerate}
\end{theorem}

\begin{proof}
(i) Let $k$ and $n$ be as in the statement of the theorem. Let $L \in
\mathfrak{B}_{m}$, $m \le n$, with $b_{1}, \ldots, b_{m}$ join-irreducible
elements that join to $1$. In order to show that $\vph_k$ holds in $L$, we
take any sequence $a_i$ of $k$ elements in $L$ and show that $\vph_k$
evaluates to $0$ for $p_i=a_i$. If there are $i\neq j$ such that $\neg a_i$
and $\neg a_j$ are comparable then the first clause of $\vph_k$ is satisfied.
So suppose that all $\neg a_i$ are pairwise incomparable. We have to show
that then the last clause of $\vph_k$ is satisfied, i.e.\ that $\neg(\neg a_1
+ \ldots + \neg a_k) = 0$, or equivalently, $\sum_{i=1}^k \neg a_i =1$. By
Lemma~\ref{lemman} every $\neg a$ is of the form $\neg a = \Join_{i\in I}
b_i$. Note that $\Join_{i\in I} b_i \leq \Join_{j\in J} b_j$ if and only if
$I\subseteq J$, as follows from (\ref{incomp}). So to the $k$ incomparable
negations $\neg a_i$ corresponds a collection of $k$ pairwise
$\subseteq$-incomparable subsets of $\{1,\ldots, m\}$. Sperner's Theorem says
that $\binom{m}{\lfloor m/2\rfloor}$ is the maximum number $k$ for which
there is such an antichain of $k$ pairwise incomparable subsets of
$\{1,\ldots, m\}$. Hence because $\binom{m}{\lfloor m/2\rfloor} \leq k$, the
collection corresponding to the $\neg a_i$ covers all of $\{1,\ldots, m\}$,
and in particular
$$
\sum_{i=1}^k \neg a_i = \sum_{i=1}^m b_i =1,
$$
which is what we had to prove.

(ii)
Suppose that $L \in \mathfrak{B}_{m}^{\bot}$, with $m>n$: let
\[
I=\{b_1, \ldots, b_n, b_{n+1}, \ldots, b_{m}\}
\]
be an antichain of join-irreducible elements such that in $L$ we have
$1=\sum_{1 \le i \le m}b_{i}$. By Sperner's Theorem take a collection of $k$
incomparable subsets $\{I_i: 1\le i \le k\}$ of $\{1,\ldots, n\}$. For every
$i=1, \ldots, k$ choose $a_i$ so that $\neg a_i = \Join_{j\in I_i} b_j$. (The
proof of Lemma~\ref{lemman} shows how to achieve this: take $a_i=\sum_{j
\notin I_i}b_j$.) Then the negations $\neg a_i$ are incomparable because the
sets $I_i$ form an antichain, and hence the first clause of $\vph_k$ is
nonzero (as $0$ is meet-irreducible in $L$). We also have
\[
\sum_{i=1}^{k} \neg a_i = \sum_{\substack{1\le i \le k\\ j\in I_{i}}} b_j \neq 1
\]
(because no $b_{j}$, with $j>n$, is included), hence
$\neg(\Join_{i=1}^k \neg a_i)\neq 0$ and the second clause of $\vph_k$ is
also nonzero. So $\phi_k$ does not evaluate to $0$ in $L$, since in this
algebra, $0$ is meet-irreducible.
\end{proof}

\section{An application to the Medvedev lattice }\label{sec:addendum}

This section is an addendum to~\cite{SorbiTerwijn}. We thank Paul Shafer
\cite{Shafer} for pointing out some inaccuracies in that paper. In
\cite{SorbiTerwijn} logics of the form $\Th(\M/\mathbf{A})$ are studied,
where $\M$ is the Medvedev lattice, $\mathbf{A}\in \M$, and $\M/\mathbf{A}$
is the initial segment of $\M$ consisting of all $\mathbf{B}\in \M$ such that
$\mathbf{B}\le \mathbf{A}$. The Medvedev lattice arises from the following
reducibility on subsets of $\omega^\omega$ (also called \emph{mass
problems}): if $\mathcal{A}, \mathcal{B}$ are mass problems, then
$\mathcal{A}\le \mathcal{B}$, if there is an oracle Turing machine which,
when given as oracle any function $g \in \mathcal{B}$, computes a function
$f\in \mathcal{A}$. The \emph{Medvedev degrees}, or simply, \emph{M-degrees},
are the equivalence classes of mass problems under the equivalence relation
generated by $\le$. The collection of all M-degrees constitutes a bounded
distributive lattice, called the \emph{Medvedev lattice}, which turns out to
be in fact a Brouwer algebra, i.e.\ it is equipped with a suitable operation
$\rightarrow$, satisfying \eqref{eqn:arrow}. Hence every factor of the form
$\M/\mathbf{A}$ is itself a Brouwer algebra, being closed under
$\rightarrow$, with $\neg$ given by $\neg \mathbf{B}=\mathbf{B} \rightarrow
\mathbf{A}$. In the following we use the notation from~\cite{SorbiTerwijn},
to which the reader is also referred for more details and information about
the Medvedev lattice and intermediate propositional logics.

In order to show that there are infinitely many logics of the form
$\Th(\M/\mathbf{A})$, in~\cite{SorbiTerwijn} a sequence of M-degrees
$\mathbf{B}_n$, $n\in\omega$, is introduced. In Corollary 5.8 of
\cite{SorbiTerwijn} it is claimed that the logics $\Th(\M/\mathbf{B}_n)$ are
all different but no detailed proof of this is given. Below we prove that
indeed these logics are all different from each other. In particular for any
$f\in\omega^\omega$ consider the mass problem
$$
\B_f = \bigset{g\in\omega^\omega : g\not\leq_T f }:
$$
then the Medvedev degree $\mathbf{B}_f$ of $\B_f$ is join-irreducible,
\cite{Sorbi:Brouwer}.
Recall that the top element $1$ of $\M/\mathbf{B}_n$ is the join
$$
\mathbf{B}_n=\mathbf{B}_{f_1}\join \ldots \join \mathbf{B}_{f_n}
$$
where $\bigset{f_i: i \in \omega}$ is a collection of functions whose Turing
degrees are pairwise  incomparable. In particular, the top element of
$\M/\mathbf{B}_1$ is join-irreducible and the top elements of all other
factors $\M/\mathbf{B}_n$ are not. Hence $\Th(\M/\mathbf{B}_1)$ can be
distinguished from all the other theories by the formula \eqref{wlem}.
Namely, the \wlem\ holds in a factor $\M/\mathbf{A}$ if and only if
$\mathbf{A}$ is join-irreducible, cf.\ \cite{Sorbi:Quotient}.
We recall that the least element of $\mathfrak{M}$, and thus of every
factor $\M/\mathbf{A}$, is meet-irreducible.
Hence $\M/\mathbf{B}_{n}\in \mathfrak{B}_n^{\bot}$.
(This is in fact enough for the proof below.)

\begin{corollary} \label{cor}
If $m\ne n$ then $\Th(\M/\mathbf{B}_{m}) \ne \Th(\M/\mathbf{B}_{n})$.
\end{corollary}
\begin{proof}
Assume $n<m$, and let $k=\binom{n}{\lfloor n/2\rfloor}$. Since
$\M/\mathbf{B}_{n}\in \mathfrak{B}_n^{\bot}$, by Theorem~\ref{thm:validity},
we have that $\vph_{k} \in \Th(\M/\mathbf{B}_{n})$, but $\vph_{k} \notin
\Th(\M/\mathbf{B}_{m})$. Notice also that by
Corollary~\ref{cor:characterization}, we can now also conclude that
$\sigma_{n} \in \Th(\M/\mathbf{B}_{n})$, but  $\sigma_{n} \notin
\Th(\M/\mathbf{B}_{m})$
\end{proof}

\section{Acknowledgements}

Thanks to Paul Shafer for his comments on the paper \cite{SorbiTerwijn}. We
thank Lev Beklemishev for remarks about $\KC$, Wim Veldman for the reference
to Brouwer, and Rosalie Iemhoff for general discussions about $\IPC$.


\begin{thebibliography}{10}

\bibitem{Proofs} M.~Aigner and G.~M. Ziegler.
\newblock {\em Proofs from The Book}.
\newblock Springer-Verlag, Berlin Heidelberg New York, 3 edition, 2004.

\bibitem{Balbes-Dwinger:Book} R.~Balbes and P.~Dwinger.
\newblock {\em Distributive Lattices}.
\newblock University of Missouri Press, Columbia, 1974.

\bibitem{Brouwer1964} L.~E.~J. Brouwer.
\newblock Consciousness, {P}hilosophy, and {M}athematics.
\newblock In E.~W. Beth, H.J. Pos, and H.~J~.A. Hollak, editors, {\em
  Proceedings of the 10th International Congress of Philosophy, August 1948,
  Amsterdam. ({R}eprinted in: {P}. {B}enecerraf and {H}. {P}utnam (eds.),
  {P}hilosophy of {M}athematics, selected readings, {P}rentice-{H}all, 1964.)},
  volume~I. North-Holland, 1948.

\bibitem{ChagrovZakharyaschev} A.~Chagrov and M.~Zakharyaschev.
\newblock {\em Modal Logic}, volume~35 of {\em Oxford Logic Guides}.
\newblock Oxford University Press, Oxford, 1997.

\bibitem{DummettLemmon} M.~A.~E. Dummett and E.~J. Lemmon.
\newblock Modal logics between $s4$ and $s5$.
\newblock {\em Z.\ Math.\ Logik Grundlag.\ Math.}, 5:250--264, 1959.

\bibitem{Fitting} M.~C. Fitting.
\newblock {\em Intuitionistic Logic, Model Theory and Forcing}.
\newblock Studies in Logic and the Foundations of Mathematics Vol.\ 21.
North-Holland, Amsterdam, 1969.

\bibitem{Gabbay} D.~M. Gabbay.
\newblock {\em Semantical Investigations in Heyting's Intuitionistic Logic},
  volume 148 of {\em Studies in Epistemology, Logic, Methodology, and
  Philosophy of Science}.
\newblock D. Reidel, Dordrecht, Boston, London, 1981.

\bibitem{Harrop1958} R.~Harrop.
\newblock On the existence of finite models and decisions procedures for
  propositional calculi.
\newblock {\em Proc. Cambridge Philos. Soc.}, 54:1--13, 1958.

\bibitem{Jankov} A. V. Jankov,
\newblock Calculus of the weak law of the excluded middle.
\newblock {\em Izv. Akad. Nauk SSSR}, Ser. Mat. 32 (1968) 1044--1051. (In Russian.)

\bibitem{Ono-Kripke} I.~Ono.
\newblock Kripke models and intermediate logics.
\newblock {\em Publ. Res. Inst. Math. Sci.}, 6:461--476, 1970.

\bibitem{Shafer} P.~Shafer.
\newblock email correspondence, May 2009.

\bibitem{Smorynski} C.~Smorynski.
\newblock {\em Investigations of Intuitionistic Formal Systems by Means of
  Kripke Models}.
\newblock PhD thesis, University of Illinois at Chicago, 1973.

\bibitem{Sorbi:Brouwer} A.~Sorbi.
\newblock Embedding {B}rouwer algebras in the {M}edvedev lattice.
\newblock {\em Notre Dame J. Formal Logic}, 32(2):266--275, 1991.

\bibitem{Sorbi:Quotient} A.~Sorbi.
\newblock Some quotient lattices of the {M}edvedev lattice.
\newblock 37:167--182, 1991.

\bibitem{SorbiTerwijn} A.~Sorbi and S.~A. Terwijn.
\newblock Intermediate logics and factors of the {M}edvedev lattice.
\newblock {\em Ann.\ Pure Appl.\ Logic}, 155(2):69--86, 2008.

\bibitem{Sperner1928} E.~Sperner.
\newblock Ein {S}atz \"uber {U}ntermengen einer endlichen {M}enge.
\newblock {\em Math. Zeitschrift}, 27:544--548, 1928.

\end{thebibliography}
\end{document}